\documentclass[12pt]{article}
\usepackage{amsfonts}
\usepackage{amsmath}
\usepackage{amsthm}
\usepackage[dvips]{graphicx}

\newtheorem{theorem}{Theorem}[section]
\newtheorem{proposition}{Proposition}[section]
\newtheorem{lemma}{Lemma}[section]
\newtheorem{corollary}{Corollary}[section]
\newtheorem{example}{Example}[section]

\newcommand{\RR}{\mathbb{R}}
\newcommand{\Rp}{\RR_+}

\newcommand{\Rpnn}{\RR^{n\times n}_+}

\def\Sat{\operatorname{Sat}}
\def\red{\operatorname{red}}

\begin{document}



\title{
On the Markov Chain Tree Theorem in the Max Algebra}

\author{
Buket Benek Gursoy\thanks{Hamilton Institute, National University of Ireland,
Maynooth, Maynooth, Co. Kildare, Ireland (buket.benek@nuim.ie, oliver.mason@nuim.ie). Supported by the Irish 
Higher Educational Authority (HEA) PRTLI Network Mathematics Grant.
}
\and
Steve Kirkland\thanks{Hamilton Institute, National University of Ireland,
Maynooth, Maynooth, Co. Kildare, Ireland (stephen.kirkland@nuim.ie). 
Research supported in part by the Science Foundation Ireland under
Grant No. SFI/07/SK/I1216b.
}
\and 
Oliver Mason\footnotemark[1]
\and 
Serge{\u\i} Sergeev\thanks{University of Birmingham, School of Mathematics, UK, Edgbaston B15 2TT (sergiej@gmail.com). Supported by EPSRC grant RRAH15735, RFBR grant 12-01-00886 and joint
RFBR-CNRS 11-01-93106.
}}

\maketitle

\begin{abstract}
The Markov Chain Tree Theorem is extended to the max algebra and possible applications to ranking problems are discussed.


\end{abstract}
\section{Introduction} \label{intro-sec}

The Matrix Tree Theorem for Markov chains (referred to as the Markov Chain Tree Theorem) is a well-known result that relates the stationary distribution of an irreducible Markov chain with the weights of directed spanning trees of its associated digraph. For a directed graph $D = (V, E)$ and $1 \leq i \leq n$, a spanning subgraph $T=(V, E)$ of $D$ is said to be an $i$-tree if the following conditions are satisfied:
\begin{itemize}
\item[(i)] for every $j \neq i$ in $\{1, \ldots , n\}$, there is exactly one outgoing edge $e \in E$ whose beginning node is $j$;
\item[(ii)] there is no edge $e \in E$ whose beginning node is $i$;
\item[(iii)] the subgraph $(V, E)$ contains no directed cycle.
\end{itemize}

We now recall the classical Markov Chain Tree Theorem.  $D(A)$ denotes the weighted directed graph associated with an irreducible matrix $A \in \Rpnn$; $D(A)$ consists of the nodes $\{1, \ldots, n\}$ with a directed edge $(i, j)$ from $i$ to $j$ of weight $a_{ij}$ if and only if $a_{ij} > 0$.  We say the edge $e=(i, j)$ is outgoing from $i$ and write $t(e) = i$.  Given an $i$-tree $T$ in $D(A)$, the \emph{weight} of $T$ is given by the product of the weights of the edges in $T$ and is denoted by $\pi(T,A)$ or just
by $\pi(T)$ when $A$ is clear from the context.  

For $1 \leq i \leq n$, define $\mathcal{T}_i$ to be the set of all $i$-trees of $D(A)$.  The classical Matrix Tree Theorem for Markov chains, also known as the  Fre\u{\i}dlin-Wentzell formula~\cite{FW84, Son99}, can be stated as follows.

\begin{theorem} 
\label{thm:matrixtreethm} Let $A \in \Rpnn$ be an irreducible (row) stochastic matrix.  Define $w \in \RR^n_+$ by 
$$w_i = \sum_{T \in \mathcal{T}_i} \pi(T).$$
Then $A^Tw = w$.  In particular, $\frac{w}{\sum_{i=1}^n w_i}$ is the unique stationary distribution of the Markov chain with transition matrix $A$.
\end{theorem}

This core result has appeared in a variety of different contexts~\cite{FW84, AT89, Bro89, Ald90, Wic09}. It was discovered by Shubert~\cite{Shu75} in connection with flow-graph methods, and independently by Kohler-Vollmerhaus~\cite{KV80} motivated by problems in biological modelling. For another reference which discusses its extension to general, not necessarily irreducible Markov chains, see Leighton-Rivest~\cite{LR83}.  

One of the primary contributions of this paper is to extend the Matrix Tree Theorem for Markov chains to the setting of the max algebra.  We show this in two ways; first, we prove a max-algebraic version of the Markov Tree Theorem directly; we then provide an alternative proof using dequantization.   We also describe some specific results in connection with the max-algebraic spectral theory. In keeping with Bapat~\cite{Bap98}, the max algebra consists of the non-negative real numbers equipped with the two operations $a \oplus b = \max(a, b)$ and $a \otimes b = ab$. These operations extend to nonnegative matrices and vectors in the standard way \cite{Bap98, Cun79, BCOQ92, But10}. 

The layout of the paper is as follows. In Section \ref{sec:result}, we obtain the Matrix Tree theorem in the max algebra, and we also give an alternative proof using dequantization. In Section \ref{sec:Kleene}, we show how to associate our main result with the max-algebraic spectral theory. In particular, we consider the connection with the Kleene star of an irreducible max-stochastic matrix. In Section \ref{sec:app}, we discuss the possibility of applying these results to decision making problems. Finally, in Section \ref{sec:conc}, we present our conclusions and future prospects.

\section{Markov Chain Tree Theorem}
\label{sec:result}
In this section, we first show that Theorem \ref{thm:matrixtreethm} extends to the max algebra. We then provide a second alternative proof of this result using dequantization.  

\subsection{Markov Chain Tree Theorem in Max Algebra}
\label{sub:maxalg}
In the main result below, we present a max-algebraic version of the Matrix Tree Theorem for Markov chains. 

Let us first recall standard observations on graphs and spanning trees.

\begin{lemma}
\label{l:tree}
Let $D$ be a digraph and $i$ be a node of $D$, to which every other node can be connected by a path. Then $D$ contains an $i$-tree.
\end{lemma}

\begin{corollary}
\label{c:tree}
If $A\in\RR^{n\times n}_+$ is irreducible then for each node $i\in\{1,\ldots,n\}$ there exists an $i$-tree in $D(A)$ with nonzero weight. 
\end{corollary}

\begin{lemma}
\label{l:path} Let $D$ be a digraph, $i$ be a node of $D$ and $T$ be an $i$-tree. Then for each node $j \neq i$ of $D$, there exists a unique directed path from $j$ to $i$ in $T$.
\end{lemma}

We now consider an irreducible matrix $A$ in $\RR^{n \times n}_+$ which is row stochastic in a max-algebraic sense.  Formally, we assume that for $1 \leq i \leq n$, $\max\limits_{1 \leq j \leq n} a_{ij} = 1$ or using max-algebraic notation
$$A\otimes \mathbf{1} = \mathbf{1}.$$
In a convenient abuse of notation, we refer to matrices satisfying the above condition as \textit{max-stochastic}.  Our main result shows that Theorem \ref{thm:matrixtreethm} extends in a natural way to the max-algebra.

\begin{theorem}
\label{thm:maxtreethm} Let $A \in \RR^{n \times n}_+$ be an irreducible max-stochastic matrix.  Define the vector $w$ by 
\begin{equation}
\label{eq:mtt2}
w_i = \bigoplus_{T \in \mathcal{T}_i} \pi(T), \;\; 1 \leq i\leq n.
\end{equation}
Then $$A^T \otimes w = w.$$
\end{theorem}
\begin{proof}
We first show that $A^T\otimes w\leq w$. To this end, let an arbitrary $i\in\{1,\ldots,n\}$ be given.
Then as $a_{ii}\leq 1$, it is immediate that $a_{ii}w_i\leq w_i$. Now consider $j\neq i$ such that
$a_{ji}\neq 0$. 
Let $T_j$ be a $j$-tree such that $\pi(T_j)=w_j$, let $E_j$ be the set of edges of $T_j$ and
$(i,k)\in E_j$. Consider the set of edges formed by removing $(i,k)$ from $E_j$ and inserting 
$(j,i)$ instead, and denote it by $E_i$. Consider the subgraph $T_i = (V, E_i)$.  Note that there is exactly 
one outgoing edge from every $\ell \neq i$ and no outgoing edge from $i$.  Further, $T_i$ is acyclic as any cycle in $T_i$ must contain the edge $(j, i)$ (otherwise it would define a cycle in the original $j$-tree $T_j$); however there is no outgoing edge from $i$ in $T_i$. It follows that the graph $T_i$ is an $i$-tree. By construction and since all entries of a max-stochastic
matrix are not greater than $1$, we obtain that
$$
w_i\geq \pi(T_i)=\pi(T_j)a_{ji}/a_{ik}\geq w_j a_{ji},
$$
and since we were given an arbitrary $i$ and took an arbitrary $j$ such that $a_{ji}\neq 0$,
it follows that $A^T\otimes w\leq w$.

To complete the proof, we show that $A^T\otimes w\geq w$. Let an arbitrary $i\in\{1,\ldots,n\}$ be given, and let
$T_i$ be an $i$-tree such that $\pi(T_i)=w_i$. As $A$ is a max-stochastic matrix by assumption, we know that $a_{ik}=1$ for
some $k$. If $k=i$ then $(A^T\otimes w)_i\geq a_{ii}w_i=w_i$. So let $k\neq i$. 
To show that $\bigoplus\limits_{j=1}^n w_j a_{ji}\geq w_i$ we will construct a $j$-tree $T_j$ such that 
$\pi(T_j)a_{ji}=\pi(T_i)$. Consider a path connecting $k$ to $i$ in $T_i$. By Lemma~\ref{l:path} this path is unique. Let
$j$ be the penultimate node on this path, meaning that $(j,i)\in E_i$. Removing the edge $(j,i)$ from $E_i$ and inserting the
edge $(i,k)$ we obtain the edge set $E_j$ and the required $j$-tree $T_j=(V,E_j)$. 
Indeed, there is exactly one outgoing edge from each node other than $j$ in $T_j$, and there is no outgoing edge from $j$.  Furthermore, if there exists a cycle in $T_j$, it must contain the edge $(i,k)$ as otherwise it would define a cycle in $T_i$.  This would then imply that there exists a directed path in $T_j$ from $k$ to $i$, all of whose edges are also edges in $T_i$.  This is impossible however, as the only such path in $T_i$ contains the edge $(j, i)$ which is not an edge in $T_j$.  Therefore $T_j$ is indeed a $j$-tree,
which satisfies $\pi(T_j)a_{ji}=\pi(T_i)$ by construction. Hence $\bigoplus\limits_{j=1}^n w_j a_{ji}\geq w_i$ and $A^T\otimes w\geq w$, as $i$ was arbitrary.
The proof is complete.
\end{proof}

The vector $w$ defined in (\ref{eq:mtt2}) in Theorem~\ref{thm:maxtreethm} will be called the {\em maximal RST (Rooted Spanning Tree) vector} of $A$. 

\subsection{Proof by dequantisation}
In this subsection, we present an alternative proof of Theorem~\ref{thm:maxtreethm} using a procedure that can be seen as an instance of the Maslov dequantization~\cite{LM98}. Note that the same procedure was used by  Olsder and Roos~\cite{OR88} to derive max-algebraic analogues of the Cramer and Cayley-Hamilton formulae.

For $1 \leq p < \infty$, consider the set of nonnegative numbers $\RR_+$ equipped with  the operations $a+_p b:=(a^p+b^p)^{1/p}$ and $a\times_p b:=ab$. For $ 1\leq p < \infty$, this is a semiring isomorphic to the semiring of nonnegative numbers with the usual arithmetic, via the mapping $f(a):=a^{1/p}$.  We denote by $\Rp(\max)$ the semiring of nonnegative real numbers equipped with the operations $\oplus$, $\otimes$ defined above. We say that $A\in\RR_+^{n\times n}$ is $p$-stochastic if $a_{i1}+_p a_{i2}+_p\cdots +_p a_{in}=1$ for $1\leq i\leq n$.

The RST vector of $A \in \mathbb{R}^{n \times n}_+$ defined as in Theorem~\ref{thm:matrixtreethm} using the arithmetics of $\Rp(p)$ will be denoted by $w^{(p)}(A)$, and when defined in $\Rp(\max)$ (i.e., the maximal RST vector), by $w^{\max}(A)$. 


\begin{theorem}
\label{t:dequant}
Let $A\in\Rpnn$ be max-stochastic.  There exists an integer $P_0$ and a sequence $A^{(p)}, p \geq P_0$ in $\Rpnn$, where each $A^{(p)}$ is $p$-stochastic, such that $\lim\limits_{p\to\infty} A^{(p)}=A$ and
$\lim\limits_{p\to\infty} w^{(p)} (A^{(p)})= w^{\max}(A)$.
\end{theorem}
\begin{proof}
Let $B^{\delta}(A)$ denote the set of matrices $C$ such that $|c_{ij}-a_{ij}|\leq\delta$ for all $i,j$
and such that $c_{ij}>0$ if and only if $a_{ij}>0$. We start by constructing a nondecreasing
sequence of $p$-stochastic matrices $A^{(p)}\in B^{\delta}(A)$. 

As $A$ is max-stochastic, for each $i\in\{1,\ldots,n\}$ there are
$l_i$ entries $a_{ij}=1$, where $0<l_i\leq n$.  We denote the other entries in each row by
$J_i:=\{j\mid a_{ij}<1\}$ for each $i$.  Choose $P_0$ so that 
$$1 - \sum_{j \in J_i} a_{ij}^p \geq 0$$
for all $p \geq P_0$.  Then for $p \geq P_0$, define $A^{(p)}$ by
$$a_{ij}^{(p)}=
\begin{cases} 
a_{ij}, & \text{if $a_{ij}<1$},\\ 
\delta_i, & \text{otherwise}, 
\end{cases} $$
where
$$\delta_i=\left(
\frac{1-\sum_{j\in J_i} a_{ij}^p}{l_i}
\right)^{1/p}.$$
It is readily verified that $A^{(p)}$ $p$-stochastic and that $a_{ij}^p \leq a_{ij}$ for all $i, j$.

Denoting $m_i=\max\{a_{ij}\mid j\in J_i\}$ we obtain
\begin{equation}
\label{e:bound}
a_{ik}-a_{ik}^{(p)}\leq 
1-\delta_i\leq 1-
\left(
\frac{1-m_i^p(n-l_i)}{l_i}
\right)^{1/p}
\end{equation}
for all $i$ and $k$.  As $m_i < 1$ for all $i$, it follows that the right hand side of~\eqref{e:bound} converges to $0$ as $p$ tends to infinity. Hence $A^{(p)}$ converges to $A$. 

Next, note that
\begin{align}
\label{eq:boundsp1}
|w_i^{(p)}(A^{(p)})-w^{\max}_i(A)| & = |w_i^{(p)}(A^{(p)})-w^{\max}_i(A^{(p)})+w^{\max}_i(A^{(p)})-w_i(A)|\displaybreak[0]\\\nonumber
 & \leq |w_i^{(p)}(A^{(p)})-w^{\max}_i(A^{(p)})|+|w^{\max}_i(A^{(p)})-w^{\max}_i(A)|.  \nonumber
\end{align}

Since $w^{\max}_i(A^{(p)})=\max\limits_{T\in\mathcal{T}_i}\pi(T, A^{(p)})$, we see that
\begin{equation}
\label{eq:boundsp2}
w_i^{(p)}(A^{(p)})-w^{\max}_i(A^{(p)})\leq (M_i^{1/p}-1)\max\limits_{T\in\mathcal{T}_i}\pi(T, A^{(p)})\leq M_i^{1/p}-1
\end{equation}
where $M_i$ is the number of $i$-trees in $A$ (or $A^{(p)}$).

It is obvious from the definition of $A^{(p)}$ that $w^{\max}_i(A)\geq w^{\max}_i(A^{(p)})$. Let $T'$ be an $i$-tree such that $w^{\max}_i(A)=\max\limits_{T\in\mathcal{T}_i}\pi(T, A)=\pi(T', A)$. It follows that $w^{\max}_i(A^{(p)})\geq \pi(T', A^{(p)})$. Then
\begin{equation}
\label{eq:boundsp3}
|w^{\max}_i(A^{(p)})-w^{\max}_i(A)|  = w^{\max}_i(A)-w^{\max}_i(A^{(p)})\leq  \pi(T', A)-\pi(T', A^{(p)}).
\end{equation}
Let $E'=\{(i_1, j_1), (i_2, j_2), ..., (i_{n-1}, j_{n-1})\}$ be the edges in the $i$-tree $T'$. Then, it follows from (\ref{eq:boundsp3}) that
\begin{align}
\label{eq:boundsp4}
|w^{\max}_i(A^{(p)})-w^{\max}_i(A)| & \leq a_{i_ij_1}a_{i_2j_2} \ldots a_{i_{n-1}j_{n-1}}-a_{i_{1}j_{1}}^{(p)}a_{i_{2}j_{2}}^{(p)} \ldots a_{i_{n-1}j_{n-1}}^{(p)} \displaybreak[0]\\\nonumber
 & \leq C(A) \max\limits_{i, j} (a_{ij}-a_{ij}^{(p)}) \nonumber
\end{align}
where $C(A)$ is a fixed constant that depends only on the entries of $A$.

Using (\ref{eq:boundsp2}) and (\ref{eq:boundsp4}) in (\ref{eq:boundsp1}), we obtain
$$|w_i^{(p)}(A^{(p)})-w^{\max}_i(A)| \leq M_i^{1/p}-1 + \max\limits_{i, j} (a_{ij}-a_{ij}^{(p)}).$$

As we showed above that $\max(a_{ij}-a_{ij}^{(p)})\to 0$ as $p\to\infty$ and $M_i^{1/p}\to 1$ as $p\to\infty$, the claim follows.
\end{proof}

As each of the semirings $\Rp(p)$ is isomorphic to the nonnegative real numbers with the usual operations, it follows from the classical Markov Chain Tree Theorem \ref{thm:matrixtreethm} that $(A^{(p)})^T\times_p w^{(p)} = w^{(p)}$ for all $p \geq P_0$.  Passing to the limit and applying Theorem \ref{t:dequant} yields another proof of Theorem~\ref{thm:maxtreethm}.

We next present some numerical examples to illustrate Theorem~\ref{thm:maxtreethm}.

\begin{example}
{\rm 
\label{ex:ex1}
\begin{equation}
\label{mat:ex1}
A=\left [ \begin{array}{cccc}
1 & 3/4& 5/6&0\\
1/2 &1 &1/4 &9/10\\
0 & 0& 1&7/8\\
1/3 &0 & 1&4/5\\
\end{array} \right ]  
\end{equation}

\begin{figure}[ht]
\begin{center}
\includegraphics[scale=0.5]{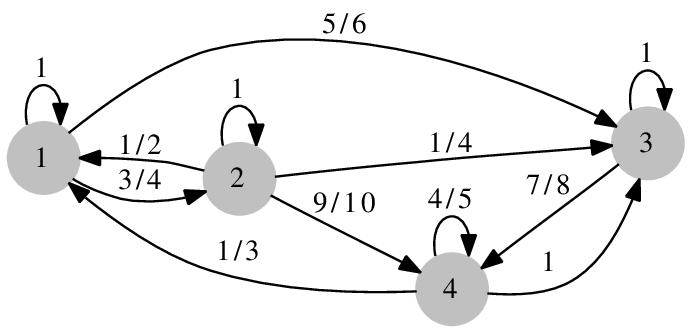}
\caption{$D(A)$ for (\ref{mat:ex1})}
\end{center}
\label{fig:Example1}
\end{figure}

Let $T_i$ be an $i$-tree with the maximum weight for $i=1, 2, 3, 4.$ Then, 
\begin{itemize}
\item{$T_1:(3,4), (2,4), (4,1)$\quad $w_1=\pi(T_1)=a_{34}a_{24}a_{41}=21/80$}
\item{$T_2:(3,4), (4,1), (1,2)$\quad $w_2=\pi(T_2)=a_{34}a_{41}a_{12}=7/32$}
\item{$T_3:(1,3), (2,4), (4,3)$\quad $w_3=\pi(T_3)=a_{13}a_{24}a_{43}=3/4$}
\item{$T_4:(2,4), (1,3), (3,4)$\quad $w_4=\pi(T_4)=a_{24}a_{13}a_{34}=21/32$}
\end{itemize}

Hence, $w=\left [ \begin{array}{c} 21/80 \\ 7/32\\ 3/4 \\ 21/32 \end{array} \right ]$  and $A^T\otimes w=w$.
}
\end{example}

\section{Maximal RST vector and Kleene star}
\label{sec:Kleene}
We have seen that the maximal RST vector $w$ associated  with the directed graph $D(A)$ is always a left max eigenvector of an irreducible max-stochastic matrix $A$.  However, in contrast to the conventional algebra, the irreducibility of $A$ is not sufficient to guarantee uniqueness (up to scalar multiple) of the max eigenvector.  This naturally leads to the question of how to identify the  maximal RST vector using the tools of max spectral theory such as the power method or Kleene star.  We next consider this question. 

First, recall that for $A\in \RR_{+}^{n\times n}$ with $\mu(A)\leq 1$ the series $I\oplus A\oplus A_{\otimes}^2\oplus ... \oplus A_{\otimes}^n\oplus ...$ converges to a finite matrix called the Kleene star of $A$ given by
$$A^* = I\oplus A\oplus A_{\otimes}^2\oplus ... \oplus A_{\otimes}^{n-1}$$
where $\mu(A)\leq 1$ \cite{Cun79, BCOQ92, But10, Car71, OW06}. Here, $A_{\otimes}^k$ denotes the $k^{\text{th}}$ max-algebraic power of $A$ and $a^*_{ij}$ is the maximum weight of a path from $i$ to $j$ of any length in $D(A)$ (if $i\neq j$). In particular if $A$ is irreducible, then $A^*$ is positive \cite{Cun79, BCOQ92}. 

A cycle with the maximum cycle geometric mean is called a critical cycle \cite{Bap98, Cun79, BCOQ92, But10, Car71, OW06}. The set of nodes that lie on some critical cycle are said to be critical nodes and denoted by $N^C(A)$. The set of edges belonging to critical cycles are said to be critical edges and denoted by $E^C(A)$. The critical matrix of $A$ \cite{ED99, ED01}, $A^C$, is formed from the submatrix of $A$ consisting of the rows and columns corresponding to critical nodes as follows. Set $a^C_{ij} = a_{ij}$ if $(i, j)$ lies on a critical cycle and $a^C_{ij} = 0$ otherwise. Moreover, we use the notation $D^C(A)$ for the critical graph of $A$, the digraph which consists of all critical nodes and edges.

The following well-known result shows the connection of $A^*$ with the max eigenvectors of $A$ \cite{Bap98, BCOQ92, ED01}. We adopt the notation $A^*_{i.}$ for the $i^{\text{th}}$ row, and the notation $A^*_{.i}$ for the $i^{\text{th}}$ column of the matrix $A^*$. 
\begin{proposition}
\label{pro:KleeneStarmaxevec}
Let $A\in \RR_{+}^{n\times n}$ be an irreducible matrix with $\mu(A)=1$. Assume that $D^C(A)$ has $r$ strongly connected components. Then, the following are true.
\begin{itemize}
\item[(i)] $\mu(A)=1$ is the only max eigenvalue of $A$;
\item[(ii)] $A^*_{.i}$ is a (right) max eigenvector associated with $\mu(A)$ for $i\in N^C(A)$;
\item[(iii)] For $i, j\in N^C(A)$ ($i\neq j$), $A^*_{.i}$ and $A^*_{.j}$ are scalar multiples of each other if they belong the same strongly connected component in $D^C(A)$. 
\end{itemize}
\end{proposition}

If one takes $r$ columns of $A^*$ from different strongly connected components of $D^C(A)$, then none of them can be expressed
as a max-linear combination of the other columns. Moreover, any such set is strongly linear independent 
in the sense of~\cite{But10}. For general (reducible) matrices, $\mu(A)$ is the biggest eigenvalue.

A max-stochastic matrix has max eigenvalue $1$, and $a_{ij}\leq 1$ for all $i,j$. This implies that $\mu(A)=1$, and that
$a_{ij}=1$ for $(i,j)\in E^C(A)$. Such matrices are called {\em visualized}~\cite{SSB09}. Note that the max-stochastic matrices
have an additional property: each node has an outgoing edge with weight $1$. The spanning subgraph of $D(A)$ consisting of the edges of weight $1$ defines the saturation digraph, denoted $\Sat(A)$.

Observe that for any matrix $A$ with a positive eigenvector $x$, the matrix $B=X^{-1}AX$, where $X$ is
a diagonal matrix formed from $x$, is max-stochastic. An analogous property holds in nonnegative algebra, where it has many 
applications, and one can consider a generalization to semifields (i.e., semirings with invertible multiplication). 
Thus a max-stochastic matrix can be considered to be ``eigenvector-visualised''.

The Kleene star of a visualised matrix with $\mu(A)=1$ (and hence of a max-stochastic one) has a very specific structure,
as described, for example, in Proposition 4.1 of~\cite{SSB09}, which we now recall.
Define $D^{C*}(A)$ to be the directed graph formed by adding trivial graphs each consisting of just one non-critical
node to $D^C(A)$ (we add one such graph for each non-critical node).  We assume that $D^{C*}(A)$ has $r'$ strongly connected components
with node sets $N_1, \ldots, N_{r'}$.

For $1\leq \mu, \nu \leq r'$, denote by $A_{\mu\nu}$ the submatrix of $A$ formed from the rows with indices in $N_{\mu}$
and from the columns with indices in $N_{\nu}$. Let $A^{\red}$ be the $r'\times r'$ matrix with entries 
$\alpha_{\mu\nu}=\max \{a_{ij}\mid i\in N_{\mu},\, j\in
N_{\nu}\}$, and let $E\in\Rpnn$ be the $n\times n$ matrix with all
entries equal to $1$.
\begin{proposition}[\cite{SSB09}, Proposition 4.1]
\label{vis-kls}
Let $A\in\Rpnn$ be a visualised matrix, $\mu(A)=1$ and $r'$
be the number of strongly connected components of $D^{C*}(A)$.
Then
\begin{itemize}
\item[1.] $\alpha_{\mu\mu}=1$ for all $1 \leq \mu\leq r'$ and $\alpha_{\mu\nu}\leq 1$
(resp. $\alpha_{\mu\nu}<1$ for $\mu\neq\nu$), where
$\mu,\nu\in\{1, \ldots, r'\}$);
\item[2.] for $1 \leq \mu,\nu \leq r'$, the corresponding submatrix of $A^*$ , 
$A^*_{\mu\nu}=\alpha_{\mu\nu}^* E_{\mu\nu}$, where $\alpha_{\mu\nu}^*$ is
the $(\mu,\nu)$-entry of $(A^{\red})^*$, and $E_{\mu\nu}$ is the
$(\mu,\nu)$-submatrix of $E$.
\end{itemize}
\end{proposition}

We proceed with the following preliminary result.
\begin{lemma}
\label{lem:KleeneStarcolsmin} Let $A \in \RR^{n \times n}_+$ be an irreducible max-stochastic matrix. Then, for $1\leq j\leq n$, $\min\limits_{1\leq i\leq n}a^*_{ij}=\min\limits_{q\in N^C(A)}a^*_{qj}$.
\end{lemma}
\begin{proof}
Let $j \in \{1, \ldots, n\}$ be given.  It is immediate that 
\begin{equation}\label{eq:KStarIn1}
\min\limits_{1\leq i\leq n}a^*_{ij}\leq \min\limits_{q\in N^C(A)}a^*_{qj}.
\end{equation}

To show the reverse inequality, consider some $l\notin {N^C(A)}$.  We claim that there exists a path from $l$ to some $k \in N^C(A)$ of weight 1.  As $A$ is max-stochastic, there exists at least one outgoing edge from $l$ of weight 1, $a_{lk_1} = 1$.  Moreover, as $l$ is not critical, $k_1 \neq l$.  If $k_1$ is critical, we are done.  If not, then there exists $k_2 \not\in \{l, k_1\}$ with $a_{k_1k_2} = 1$.  Continuing in this fashion, we must eventually arrive at some node $k = k_p$ which was already on the path.
Hence this node is on a critical cycle, and $k$ is in $N^C(A)$.  By construction, $l, k_1, \ldots, k_p = k$ is a path of weight $1$, which we denote by $P_1$.  

$a^*_{kj}$ is the maximal weight of a path $P_2$ between $k$ and $j$. Concatenation of $P_1$ and $P_2$ yields the path
$P_1\circ P_2$ with weight $a^*_{kj}$ connecting $l$ to $j$.
It follows that
$$a^*_{lj} \geq a^*_{kj} \geq \min\limits_{q\in N^C(A)}a^*_{qj}.$$
As this must hold for any $l \not \in N^C(A)$, we have that 
\begin{equation}
\label{eq:KStarIn2}
\min\limits_{1\leq i\leq n}a^*_{ij}\geq \min\limits_{q\in N^C(A)}a^*_{qj}.
\end{equation}
Combining (\ref{eq:KStarIn2}) and (\ref{eq:KStarIn1}) yields the result.
\end{proof}


Recall that for a max-stochastic matrix $A$, each node has an outgoing edge with weight $1$. The spanning 
subgraph of $D(A)$, which contains
the edge $(i,j)$ if and only if $a_{ij}=1$ is known as the saturation subgraph $\Sat(A)$ of $A$. 

\begin{lemma}
\label{lem:wcrit} Let $A \in \RR^{n \times n}_+$ be an irreducible max-stochastic matrix. Assume that $D^C(A)$ is 
strongly connected.  Let $w$ be the maximal RST vector of $A$.  Then for all $i \in N^C(A)$, $w_i = 1$.
\end{lemma}
\begin{proof}
Evidently $w_i\leq 1$ for all $i$, since it is obtained by multiplication
of the entries of $A$, all not exceeding $1$. To show the lemma we need to construct, for a given
critical node $i$, an $i$-tree in $\Sat(A)$. In $\Sat(A)$, each node is connected
to a critical node, but if $D^C(A)$ is 
strongly connected, then any such node is connected to $i$. The proof now
follows from an application of Lemma~\ref{l:tree}.
\end{proof}

In the next result, we denote by $y_C$ the critical subvector of $y$, i.e., the
subvector corresponding to indices in $N^C(A)$.

\begin{theorem}
\label{thm:maxsptrvecBound} Let $A \in \RR^{n \times n}_+$ be an irreducible max-stochastic matrix and $w$ be the maximal RST vector of $A$.  Then, the following are true.
\begin{itemize}
\item[(i)]  $w\leq \min\limits_{i \in N^C(A)} A^*_{i.}$;
\item[(ii)] If $D^C(A)$ is strongly connected then $w=\min\limits_{i \in N^C(A)} A^*_{i.}$;
\item[(iii)] If $D^C(A)$ has no more than two components then 
$w_C=(\min\limits_{i\in N^C(A)} A^*_{i.})_C$.
\end{itemize}
\end{theorem}
\begin{proof}
(i): Consider a $j$-tree $T$ ($1\leq j\leq n$), with weight $w_j$. 
There exists a path $P$ in $T$ from $i$ to $j$ for $1\leq i\leq n$, $i\neq j$, with weight $w(P)$. 
Then, $$w_j\leq w(P)\leq a^*_{ij} \text{ for all } i\in\{1, 2, ..., n\}.$$
Thus, $w_j\leq \min\limits_{1\leq i\leq n}a^*_{ij}$, or equivalently (by Lemma \ref{lem:KleeneStarcolsmin}), we have $w_j\leq \min\limits_{k \in N^C(A)}a^*_{kj}$. 

(ii): In this case the eigencone is a single ray, consisting of the multiples of
a column of Kleene star with index in $N^C$.
Let $y$ be equal to any such column. By Proposition~\ref{vis-kls}, all the components of $y_C$ equal $1$, 
and by Lemma~\ref{lem:wcrit} all the components of $w_C$ equal $1$. Hence $y=w_C$.

(iii): Let $D^C(A)$ consist of two components, with sets of nodes $N_1$ and $N_2$ respectively.
By Proposition~\ref{vis-kls}, there exist $\alpha$ and $\beta$ such that 
$A^*_{12}=\alpha E_{12}$, $A^*_{21}=\beta E_{21}$,
$A^*_{11}=E_{11}$ and $A^*_{22}=E_{22}$. Hence we need to show that
$w_i=\beta$ when $i\in N_1$ and $w_i=\alpha$ when $i\in N_2$. We will give a proof only for
$i\in N_1$, the other case being similar.

As we showed in part (i) that  $w\leq \min\limits_{i \in N^C(A)} A^*_{i.}$,
it suffices to build a spanning tree of weight $\beta$, directed to $i\in N_1$.
Consider a path $P$ of greatest weight connecting a node in $N_2$ to $i$; by Proposition~\ref{vis-kls} this weight is equal 
to $\beta$.  Let $k$ be the first node on $P$ where it leaves $N_2$ and let $l$ be the first node on $P$ where it enters $N_1$. By optimality of $P$,
only the subpath $P'$ of $P$ connecting $k$ to $l$ may have weight less than $1$, and this weight is $\beta$. Using 
Lemma~\ref{l:tree} , construct an $i$-tree in the first component of $D^C(A)$ (with node set $N_1$), and a $k$-tree in the second component of $D^C(A)$ (with node set $N_2$). This makes a spanning tree on the graph consisting of $D^C(A)$
and $P'$, directed to $i$. We need to complete this tree to an $i$-tree and having the same weight.
We can do this using the edges of $\Sat(A)$,
since all remaining nodes of $D(A)$ can be connected by a path with edges in $\Sat(A)$ either to a node of
$D^C(A)$ or to a node of $P'$. The resulting tree is directed to $i$ and has weight $\beta$.
\end{proof}

It follows immediately that if all nodes in $D(A)$ are critical and $D^C(A)$ has two strongly connected components, then the maximal RST vector $w$ is given by $\min\limits_{i} (A^*_{i.}).$
We now describe some numerical examples to illustrate the above results.

\begin{example}
\label{ex:ex3} Consider the matrix $A$ given in (\ref{mat:ex1}). There exist three strongly connected components in $D^C(A)$ and $N^C(A)=\{1, 2, 3\}$ such that $N_1^C(A)=\{1\}$, $N_2^C(A) = \{2\}$ and  $N_3^C(A)=\{3\}$. The Kleene star of $A$ is given by
$$A^*=\left [ \begin{array}{cccc}
1 &3/4 &5/6 &35/48\\
1/2 &1 &9/10 &9/10\\
7/24 &7/32 &1 &7/8\\
1/3 &1/4 &1 &1\\
\end{array} \right ].$$

The left max eigenvectors are $A^*_{1.}, A^*_{2.}$ and $A^*_{3.}$.

Then, $w \leq \min\limits_{i\in N^C(A)} A^*_{i.} = \left [ \begin{array}{c} 7/24 \\ 7/32\\ 5/6\\ 35/48 \end{array} \right ].$ 

However, $w_C\neq (\min\limits_{i\in N^C(A)} A^*_{i.})_C$ as there are three strongly connected components in $D^C(A)$.
\end{example}

\section{Application to AHP: discussion}
\label{sec:app}
The results of the previous sections relate the eigenvectors of max-stochastic matrices with the maximal weights of spanning trees in the associated directed graph.  In this section, we discuss possible applications of these results to questions related to the construction of ranking vectors in decision-making processes.  We first note the following simple observation.

\begin{proposition}
\label{pro:Irred} 
Let $A \in \RR^{n \times n}_+$, and let the diagonal matrix $D$ given by
$$d_{ii} = \max\limits_{1 \leq j \leq n} a_{ij}$$
have all entries nonzero.
Further let $w$ be the maximal RST vector for $A$.  Then 
$$A^T \otimes w = D w.$$
\end{proposition}
\begin{proof}
Let $\hat{A} = D^{-1} A$.  Then $\hat{A}$ is irreducible and max-stochastic.  For $1 \leq i \leq n$, consider a spanning tree $\hat{T}$ in $D(\hat{A})$ rooted at $i$.  It is clear that the weight of $\hat{T}$ takes the form
$$\hat{a}_{i_1j_1} \cdots \hat{a}_{i_{n-1}j_{n-1}} = \frac{1}{d_{i_1}d_{i_2} \cdots d_{i_{n-1}}} a_{i_1j_1} \cdots a_{i_{n-1}j_{n-1}}$$
where $\{i_1, \ldots, i_{n-1}\} = \{1, \ldots,  n\} \backslash \{i\}$.  In fact, it is clear that there is a bijective correspondence between spanning trees $T_i$ in $D(A)$ rooted at $i$ and spanning trees $\hat{T}_i$ rooted at $i$ in $D(A)$ with 
$$\pi(\hat{T_i}) = \frac{d_i}{\textrm{det}(D)} \pi(T_i).$$
It follows that if we write $\hat{w}$ for the maximal RST vector of $\hat{A}$, then 
\begin{equation}
\label{eq:maxweight} 
\hat{w} = \frac{D}{\textrm{det}(D)} w.
\end{equation}
As $\hat{A}$ is max-stochastic, we know from Theorem~\ref{thm:maxtreethm} that $\hat{A}^T \otimes \hat{w} = \hat{w}$.  Noting that $\hat{A}^T = A^T D^{-1}$, we can use (\ref{eq:maxweight}) to rewrite this as  
$$\frac{1}{\textrm{det}(D)} A^T \otimes w = \frac{D}{\textrm{det}(D)} w.$$
The result follows immediately.
\end{proof} 


Now suppose that $A$ is a symmetrically reciprocal matrix ({\it{SR}}-matrix), so that $a_{ij}a_{ji} = 1$ for all $i, j$.  Such matrices arise as a result of pairwise comparisons in the Analytic Hierarchy Process (AHP), which is a widely used framework for decision making.  The typical interpretation is that $a_{ij}$ indicates the relative strength (or score) of option $i$ to option $j$.  A central question in the AHP is to determine a weight vector $w$ in which $w_i$  represents the weight given to option $i$.  Saaty \cite{Saa77} suggested to take $w$ to be the Perron vector of $A$. Elsner and van den Driessche \cite{ED04, ED10} suggested selecting $w$ from the set of vectors, including the max-algebraic eigenvector, that minimises the functional 
\begin{equation}
\label{eq:errf}
e_A(x)=\max\limits_{1\leq i,j\leq n}a_{ij}x_j/x_i.
\end{equation}
Recall that the set of vectors that minimise (\ref{eq:errf}) is the subeigencone of $A$ with respect to $\mu(A)$ for an {\textit{SR}}-matrix $A$ \cite{BMS12}. 

In this context, a spanning tree in $D(A^T)$ rooted at $i$ represents an accumulation of \emph{relative} scores with respect to all other options in $\{1,  \ldots , n\}$.  With this in mind, the vector of maximal RST weights for $A^T$ is a reasonable choice of ranking vector.  From Proposition \ref{pro:Irred}, we know that $w$ must solve the generalised max-eigenvector equation 
$$A \otimes w = D w$$
where $D$ is diagonal and satisfies $d_{ii} = \max\limits_{1 \leq j \leq n} a_{ji}.$  It is worth noting that such a $w$ does not minimise the maximal relative error functional in (\ref{eq:errf}) and may give different rankings to the schemes considered there. On the other hand, it has the advantage that the maximal RST vector is unique, while the optimisation problem studied in these earlier papers may give rise to multiple rankings.  

Another scenario in which these results could be applied is as follows.  Suppose we have a set of $m$ ``judges'' and $n$ ``competitors''.  Each judge is asked to give the competitors a score between 0 and 1 with the highest ranked competitor scoring a $1$ and the others scored accordingly.  Moreover, each competitor is asked to score the judges in the same way.  The judges scores will generate a matrix $J \in \mathbb{R}_+^{m \times n}$ with a row for each judge, while the competitors' scores will generate a matrix $C \in \mathbb{R}_+^{n \times m}$ with a row for each competitor's scores.  

Consider now the matrix $\hat C = C \otimes J$.  For $1 \leq p, q \leq n$, consider the entry $\hat c_{pq} = \max\limits_{1 \leq r \leq m} c_{pr}j_{rq}$.  Each product $c_{pr}j_{rq}$ can be viewed as an indirect score given by competitor $p$ to competitor $q$ via judge $r$.  Thus the entry $\hat c_{pq}$ is the maximal such score over all judges.  It is easy to see that the matrix $\hat C$ will be max-stochastic.  The maximal RST vector $w$ associated with $D(\hat C)$ can be used to rank the competitors and Theorem~\ref{thm:maxtreethm} shows that $w$ is a max eigenvector of $\hat C^T$.  Similar remarks apply to the matrix $\hat J = J \otimes C$.  

\section{Concluding Remarks}
\label{sec:conc}
We have shown that the Markov Chain Tree Theorem extends to the max algebra.  
We have also shown that this fact follows from the classical result via dequantisation.  We have related the maximal RST vector to the entries of the Kleene star and briefly discussed some possible applications of these results to AHP
and ranking. 

In an ongoing work, we are going to generalize the Markov Chain Tree Theorem to commutative semirings,
and consider the computational complexity of computing the RST vector in this general setting.  

\bigskip


\end{document}